\documentclass[a4paper]{amsart}
\usepackage{amssymb}
\usepackage[matrix,arrow,curve,tips]{xy}\SelectTips{cm}{}
\usepackage{mathrsfs}

\newtheorem{dfn}{Definition}
\newtheorem{prop}[dfn]{Proposition}
\newtheorem{theo}[dfn]{Theorem}
\newtheorem{cor}[dfn]{Corollary}
\newtheorem{ex}[dfn]{Example}
\newtheorem{lem}[dfn]{Lemma}

\newcommand{\RR}{\mathbb{R}}
\newcommand{\CC}{\mathbb{C}}

\newcommand{\oo}{\,\mbox{-}\,}
\newcommand{\com}{\circ}
\newcommand{\setsep}{\;|\;}

\newcommand{\src}{\mathord{\mathrm{s}}}
\newcommand{\trg}{\mathord{\mathrm{t}}}
\newcommand{\mlt}{\mathord{\mathrm{mlt}}}
\newcommand{\inv}{\mathord{\mathrm{inv}}}
\newcommand{\uni}{\mathord{\mathrm{uni}}}
\newcommand{\id}{\mathord{\mathrm{id}}}
\newcommand{\pr}{\mathord{\mathrm{pr}}}

\newcommand{\Mon}{\mathord{\mathrm{Mon}}}
\newcommand{\I}{\mathord{\mathrm{I}}}
\newcommand{\K}{\mathord{\mathrm{K}}}
\newcommand{\Pc}{\mathord{\mathrm{Pc}}}
\newcommand{\fix}{\mathord{\mathrm{fix}}}
\newcommand{\GL}{\mathord{\mathrm{GL}}}
\newcommand{\Aut}{\mathord{\mathrm{Aut}}}
\newcommand{\Inn}{\mathord{\mathrm{Inn}}}
\newcommand{\sat}{\mathord{\mathrm{sat}}}
\newcommand{\C}{\mathord{\mathrm{C}}}
\newcommand{\N}{\mathord{\mathrm{N}}}
\newcommand{\U}{\mathord{\mathrm{U}}}
\newcommand{\Ort}{\mathord{\mathrm{O}}}
\newcommand{\con}{\mathord{\mathrm{c}}}
\newcommand{\tr}{\mathord{\mathrm{tr}}}
\newcommand{\UFr}{\mathord{\mathrm{UFr}}}

\newcommand{\GG}{\mathscr{G}}
\newcommand{\HH}{\mathscr{H}}
\newcommand{\KK}{\mathscr{K}}
\newcommand{\BB}{\mathscr{B}}

\newcommand{\Gpd}{\mathord{\mathsf{Gpd}}}
\newcommand{\GPD}{\mathord{\mathsf{GPD}}}
\newcommand{\LGpd}{\mathord{\mathsf{LGpd}}}
\newcommand{\LGPD}{\mathord{\mathsf{LGPD}}}
\newcommand{\Gpdinfty}{\mathord{\mathsf{Gpd^\infty}}}
\newcommand{\GPDinfty}{\mathord{\mathsf{GPD^\infty}}}
\newcommand{\Grp}{\mathord{\mathsf{Grp}}}

\newcommand{\vV}{\mathsf{V}}

\begin{document}

\title[]%
{Monodromy and faithful representability of Lie groupoids}

\author{Janez Mr\v{c}un}
\address{Department of Mathematics, University of Ljubljana,
         Jadranska 19, 1000 Ljubljana, Slovenia}
\email{janez.mrcun@fmf.uni-lj.si}

\thanks{This work was supported in part by
        the Slovenian Research Agency (ARRS)
       }

\begin{abstract}
For any topological groupoid $\GG$ and
any homomorphism $\rho$
from a locally compact Hausdorff topological
group $K$ to $\GG$,
we construct an associated
monodromy group $\Mon(\GG;\rho)$. We prove that
Morita equivalent topological groupoids
have the same monodromy groups. 
We show how the monodromy groups can be used to
test if a Lie groupoid lacks faithful representations.
\end{abstract}

\maketitle

\section{Introduction}

As we know from Peter-Weyl theorem, any compact Lie group
has a faithful finite dimensional representation.
It is a natural and a long-standing question how far can this result
be extended to Lie groupoids. The structure of a Lie groupoid
generalizes that of a Lie group, with
examples including
smooth manifolds, smooth bundles of Lie groups,
action groupoids associated to smooth
actions of Lie groups on smooth manifolds, holonomy groupoids of foliations,
and orbifold atlases.
Among Lie groupoids,
proper Lie groupoids are considered to play the role of
compact Lie groups.
Action groupoids of proper smooth
actions and orbifold atlases are all proper Lie groupoids.
In fact, an orbifold can be equivalently described
as a Morita equivalence class of
proper \'etale Lie groupoids \cite{Moerdijk2002}.

It turns out that a proper Lie groupoid admits
a faithful representation on a vector bundle of finite rank
if, and only if, it is Morita equivalent to the action groupoid
of a smooth action of a compact Lie group on a smooth manifold
(see, for example, \cite{Kalisnik,Trentinaglia}).
Effective orbifolds
are all faithfully representable (\cite{Satake},
see also \cite{HenriquesMetzler,MoerdijkMrcun2003}),
but it is still an open question whether all non-effective orbifolds
are faithfully representable as well.
On the other hand, an example of a proper regular (non-\'etale) Lie groupoid
which is not faithfully representable was found (\cite{LuckOliver},
see also \cite{Trentinaglia}).
In \cite{JelencMrcun},
an obstruction to faithful representability of proper regular Lie groupoids was given
and used to construct a large class of proper regular Lie groupoids
which are not faithfully representable. 
This is based on the observation that the ineffective isotropy
arrows $\K(\GG)$ of a proper regular Lie groupoid $\GG$ form a locally trivial
bundle of compact Lie groups over the manifold of objects of $\GG$,
and on an explicit computation
of an associated long exact sequence of homotopy groups
in which the obstruction can be found in 
the image of the boundary (monodromy) map
$$ \partial:\pi_1(\GG;x) \to \pi_0(\Aut(\K_x(\GG))),$$
for an object $x$ of $\GG$.
In this paper we generalize
this obstruction to general Lie groupoids
and construct examples of proper non-regular Lie groupoids which are not
faithfully representable.

Let $K$ be a locally compact topological group.
A homomorphism
of topological groupoids $\rho:K\to\GG$
is determined by an object $\omega(\rho)$ of $\GG$ and
a homomorphism of topological groups from $K$ to the isotropy
group $\I_{\omega(\rho)}(\GG)$ of $\GG$ at $\omega(\rho)$.
The space $\hom(K,\GG)$  of all homomorphisms of topological groupoids
from $K$ to $\GG$ has a natural left $\GG$-action along the map $\omega$
given by conjugation
and a natural right action of the group $\Aut(K)$ of automorphisms of $K$
given by composition.
For any $\rho\in\hom(K,\GG)$ we denote by $\hom(K,\GG;\rho)$ the
space of objects of the path-component
of the action groupoid $\GG\ltimes\hom(K,\GG)$ which includes $\rho$,
and define the monodromy group
$$ \Mon(\GG;\rho) $$
to be the subgroup of $\Aut(K)$ consisting of all automorphisms $\mu\in\Aut(K)$
which satisfy
$\hom(K,\GG;\rho)\com\mu=\hom(K,\GG;\rho)$.
We show that the monodromy group includes the subgroup $\Inn(K)$ of
inner automorphisms of $K$ and that
Morita equivalent topological groupoids have the same
monodromy groups.

If $K$ is a compact Lie group, if $\GG$ is a faithfully representable
Lie groupoid and if $\rho\in\hom(K,\GG)$ is a monomorphism,
we prove
that there exists a faithful character $\chi$ of the Lie group $K$ such that
$$ \chi\com\mu = \chi $$
for any $\mu\in\Mon(\GG;\rho)$ (Theorem \ref{theo-representability}).
This theorem shows that some automorphisms in a monodromy
group of $\GG$ can provide an obstruction to faithful representability of $\GG$
(Example \ref{ex-flat-bundle}).
The theorem also generalizes the obstruction given in \cite{JelencMrcun},
because for a proper regular Lie groupoid $\GG$ with an object $x$, the
image of the monodromy map $\partial:\pi_1(\GG;x) \to \pi_0(\Aut(\K_x(\GG)))$ 
is exactly the group
$$ \pi_0(\Mon(\GG;\id_{\K_x(\GG)})).$$
We show that for a special class of locally trivial bundles of
compact Lie groups, Theorem \ref{theo-representability} provides
not only necessary but also sufficient condition for 
faithful representability (Theorem \ref{theo-bundle-representability}).
Finally, we use the linearization theorem for a proper Lie groupoid $\GG$
to study the action Lie groupoid $\GG\ltimes\hom(K,\GG;\rho)$ and the image
$\Sigma(\GG;\rho)=\omega(\hom(K,\GG;\rho))$ locally. We show how closely the set
$\Sigma(\GG;\rho)$ is related to the fixed point set
of the action of the isotropy group $\I_{\omega(\rho)}(\GG)$ on a local slice.

\section{Groupoids and representations}

For the convenience of the reader and to fix the notations,
we shall use this section to recall some basic
definitions and facts about Lie groupoids
and their representations. Detailed presentation with examples
can be found in
\cite{MoerdijkMrcun2003,MoerdijkMrcun2005,Mackenzie,Kalisnik}.

\subsection{Topological and Lie groupoids}
A groupoid is a small category in which every arrow is invertible. 
In particular, a groupoid $\GG$ consists of a set $\GG_1$ of arrows and a set $\GG_0$
of objects, any arrow $g\in\GG_1$ has its source $\src(g)\in\GG_0$ and its
target $\trg(g)\in\GG_0$, the partial multiplication
$\mlt(g',g)= g'g$ is defined for any pair
of arrows $g',g\in\GG_1$ with $\src(g')=\trg(g)$,
for any object $x\in\GG_0$ there is the identity arrow $\uni(x)=1_x\in\GG_1$
and any arrow $g\in\GG_1$ has its inverse arrow $\inv(g)=g^{-1}\in\GG_1$.
We denote such a groupoid $\GG$ also by $(\GG_1\rightrightarrows\GG_0)$.
The above mentioned maps
$\src,\trg:\GG_1\to\GG_0$,
$\mlt:\GG_1\times^{\src,\trg}_{\GG_0}\GG_1\to\GG_1$,
$\uni:\GG_0\to\GG_1$ and $\inv:\GG_1\to\GG_1$ are called the structure maps
of the groupoid $\GG$. Note that the multiplication map $\mlt$ is defined
on the pull-back with respect to the maps $\src$ and $\trg$.
We say that a groupoid $\GG$ is a groupoid over $\GG_0$ and we often
denote the set $\GG_1$ simply by $\GG$.
We write
$\GG(x,x')$ for the set of arrows of $\GG$ with source $x$ and target $x'$,
while $\GG(x,\oo)=\src^{-1}(\{x\})$ and $\GG(\oo,x')=\trg^{-1}(\{x'\})$,
for any $x,x'\in\GG_0$.
The set $\GG(x,x)=\I_x(\GG)$ is a group, called the
\emph{isotropy group} of $\GG$ at $x$.

A topological groupoid is a groupoid $\GG$, together with a topology
on the set $\GG_1$ and a topology on the set $\GG_0$
such that all the structure maps are continuous.
With continuous functors as homomorphisms, topological groupoids
form a category $\Gpd$.

To make the terminology precise, we will say that a $C^\infty$-space
is a topological space with a $C^\infty$-structure of
constant dimension (given by a maximal $C^\infty$-atlas).
In particular, a $C^\infty$-space is locally Euclidean
of constant dimension without boundary and may not be Hausdorff.
A smooth manifold is a $C^\infty$-space
which is Hausdorff and second-countable.
(If one prefers to allow manifolds and Lie groups 
to have uncountably many components - as some authors do -
one could relax the assumption of second-countability
of smooth manifolds by the condition of paracompactness.)

A $C^\infty$-groupoid is a topological groupoid $\GG$,
together with a $C^\infty$-structure on the space $\GG_1$
and a $C^\infty$-structure on the space $\GG_0$
such that the source map $\src$ is a surjective submersion and
all the structure maps of $\GG$ are smooth.
A Lie groupoid is a $C^\infty$-groupoid $\GG$ in which the space
$\GG_0$ and all the source-fibers $\GG(x,\oo)$, $x\in\GG_0$,
are smooth manifolds.
(In general theory, we do not assume that a Lie groupoid is a smooth
manifold, in order to include many interesting examples.)
The isotropy groups of a Lie groupoid are Lie groups.
The $C^\infty$-groupoids form a category $\Gpdinfty$ in which homomorphisms
are smooth functors. The category $\LGpd$ of Lie groupoids
is the corresponding full subcategory of $\Gpdinfty$.

\begin{ex}\rm
(1) Any smooth
manifold $M$ can be viewed as a Lie groupoid
$(M\rightrightarrows M)$, where both the source map and the target
map are the identity.
Any Lie group $K$ can be viewed as a Lie groupoid 
$(K\rightrightarrows \ast)$ over a one-point set.

(2)
Let $\xi:\vV\to M$ be a smooth complex vector bundle (of finite rank)
over a smooth manifold $M$.
The general linear Lie groupoid on $\vV$
is the Lie groupoid $\GL(\vV)$ over $M$ of which the arrows with
source $x\in M$ and target $x'\in M$ are all linear isomorphisms
from the fiber $\vV_x=\xi^{-1}(x)$ to the fiber $\vV_{x'}=\xi^{-1}(x')$.
In the presence of a Hermitian metric on $\vV$ we may consider
also the unitary Lie groupoid $\U(\vV)$ over $M$, which is a closed subgroupoid
of $\GL(\vV)$ consisting of all linear isometries between the fibers
of $\vV$.

(3) Let $\GG$ be a Lie groupoid and let $U$ be an open subset of $\GG_0$.
Then the restriction $\GG|_U$ is the Lie groupoid over $U$
which consists of all arrows of $\GG$ with their source and their 
target in $U$.
In fact, the restriction $\GG|_U$ is an open subgroupoid of the Lie groupoid
$\GG$.
\end{ex}

\subsection{Actions}
A continuous right action of a topological groupoid $\GG$ on a topological space $P$
along a continuous map $\epsilon:P\to \GG_0$ is a continuous map
$P\times_{\GG_0}^{\epsilon,\trg}\GG_1\to P$, $(p,g)\mapsto pg$, which
satisfies the usual properties of an action.
To such a right $\GG$-action we associate the action groupoid
$P\rtimes\GG=(P\times_{\GG_0}^{\epsilon,\trg}\GG_1\rightrightarrows P)$,
a topological groupoid
in which the source map is the action, the target map is the projection and
the partial multiplication is given by
$(p',g')(p,g)=(p',g'g)$.
We write $p\GG=\{ pg \setsep g\in\GG(\oo,\epsilon(p)) \}\subset P$
for the orbit through the point $p\in P$ and $P/\GG$ for the quotient space
of orbits of the $\GG$-action. We say that a subset $X$ of $P$ is $\GG$-saturated
if it equals its saturation $\sat_\GG (X)$, which is
the union of all of those orbits of the $\GG$-action that intersect $X$.

Similarly we define the notion of a left continuous action of a topological groupoid
$\HH$ on $P$ along a continuous map $\pi:P\to\HH_0$, the associated action groupoid
$\HH\ltimes P$, the orbits $\HH p$, the space of orbits $\HH \backslash P$, the
$\HH$-saturated subsets of $P$ and the saturation $\sat_\HH (Y)$ of a subset $Y$ of $P$.

We can also consider
smooth actions of $C^\infty$-groupoids
on $C^\infty$-spaces along smooth maps.
The associated
action groupoid is in this case a $C^\infty$-groupoid.
The action groupoid associated to a Lie groupoid action
on a smooth manifold is a Lie groupoid.

Note that any topological groupoid $\GG$ acts continuously on $\GG_0$
from the left along the target map and from the right along the source map.
The orbit of $\GG$ through a point $x\in\GG_0$ is the orbit
$\GG x = x\,\GG = \trg(\src^{-1}(\{x\})) = \src(\trg^{-1}(\{x\}))$.
These actions are smooth if $\GG$ is a $C^\infty$-groupoid.
The orbits of a Lie groupoid $\GG$ are immersed submanifolds of $\GG_0$.
There are natural isomorphisms (in the category $\Gpd$, or $\Gpdinfty$)
between $\GG$, $\GG\ltimes\GG_0$ and $\GG_0\rtimes\GG$.

\subsection{Morita maps}
Let $\GG$ and $\HH$ be topological groupoids.
A topological principal $\GG$-bundle over $\HH$ is
a topological space $P$, together with a continuous right action of $\GG$ along
a continuous map $\epsilon$ and a continuous left action of $\HH$ along
a surjective continuous map $\pi$ with local sections such that
$\pi(pg)=\pi(p)$, $\epsilon(hp)=\epsilon(p)$ and
$h(pg)=(hp)g$, for all $p\in P$, $h\in\HH(\pi(p),\oo)$ and $g\in\GG(\oo,\epsilon(p))$,
and such that the map
$P\times_{\GG_0}^{\epsilon,\trg}\GG_1 \to P\times_{\HH_0}^{\pi,\pi} P$,
$(p,g)\mapsto (p,pg)$, is a homeomorphism.

As an important example,
for any homomorphism $\phi:\HH\to\GG$ there is the associated
principal $\GG$-bundle $\langle\phi\rangle=\HH_0\times_{\GG_0}^{\phi,\trg}\GG_1$ over $\HH$
equipped with the natural actions.
For any principal $\GG$-bundle $P$ over $\HH$, any point $p\in P$
induces a homomorphism of topological
groups $P_p:\I_{\pi(p)}(\HH)\to\I_{\epsilon(p)}(\GG)$, which is determined
by the equation $hp=pP_p(h)$, $h\in\I_{\pi(p)}(\HH)$.

Let $P$ be a topological principal $\GG$-bundle over $\HH$ and let $Q$
be a topological principal $\HH$-bundle over $\KK$,
where $\KK$ is another topological groupoid.
The tensor product $Q\otimes_\HH P$ of $Q$ and $P$
is the space of orbits of the pull-back
$Q\times_{\HH_0}^{\epsilon,\pi}P$ with respect to the diagonal action of $\HH$.
With the natural actions, the tensor product $Q\otimes_\HH P$
is a topological principal $\GG$-bundle over $\KK$.

An isomorphism between topological principal $\GG$-bundles $P,P'$ over $\HH$
is a homeomorphism $P\to P'$ which is equivariant for both the left and the right
action. A topological Morita map from $\HH$ to $\GG$ is an isomorphism class of topological
principal $\GG$-bundles over $\HH$. Such Morita maps form the Morita category $\GPD$
of topological groupoids, with composition induced by the tensor
product of topological principal bundles. 

For $C^\infty$-groupoids $\GG$ and $\HH$, a topological
principal $\GG$-bundle $P$ over $\HH$ is smooth
if it is equipped with a $C^\infty$-structure, both actions on $P$ are smooth, the map
$\pi:P\to\HH_0$ is a surjective submersion and the map
$P\times_{\GG_0}^{\epsilon,\trg}\GG_1 \to P\times_{\HH_0}^{\pi,\pi} P$,
$(p,g)\mapsto (p,pg)$, is a diffeomorphism.
The tensor product of two smooth principal bundles is again a smooth principal bundle.
An isomorphism between smooth principal $\GG$-bundles over $\HH$
is a diffeomorphism which is equivariant for both actions.
A smooth Morita map from $\HH$ to $\GG$ is an isomorphism class of smooth
principal $\GG$-bundles over $\HH$. Such smooth Morita maps form
the (smooth) Morita category $\GPDinfty$
of $C^\infty$-groupoids, with composition induced by the tensor
product of smooth principal bundles. 
The Morita category $\LGPD$ of Lie groupoids is the corresponding
full subcategory of the category $\GPDinfty$.

The isomorphisms in the Morita category (topological or smooth) are
called Morita equivalences.
Two topological (or $C^\infty$-) groupoids
are Morita equivalent if they are isomorphic
in the category $\GPD$ (respectively $\GPDinfty$).
A continuous (or smooth) functor $\phi:\HH\to\GG$
between topological (respectively $C^\infty$-) groupoids is a
weak equivalence if $\langle \phi\rangle$ represents a Morita equivalence
in the category $\GPD$ (respectively $\GPDinfty$).
For details, see \cite{Mrcun1996,MoerdijkMrcun2005}.

\begin{ex}\rm
(1) 
The Lie groupoid $\GL(\vV)$, associated to a smooth complex
vector bundle $\xi:\vV\to M$
over a smooth manifold $M$, is Morita equivalent to the Lie group $\GL(\vV_x)$,
for any $x\in M$. The natural inclusion
$\GL(\vV_x)\to\GL(\vV)$ is a weak equivalence.

(2) Let $\GG$ be a Lie groupoid and let $(U_i)_{i\in I}$
be a countable open cover of $\GG_0$.
The natural map $f:\coprod_i U_i \to \GG_0$ is a surjective local diffeomorphism,
and we may define the pull-back Lie groupoid
$f^\ast\GG$ over $\coprod_i U_i$ with arrows given by the pull-back
$(f^\ast\GG)_1=(\coprod_i U_i \times \coprod_i U_i)
 \times_{\GG_0\times\GG_0}^{f\times f,(\src,\trg)} \GG_1$.
The projection $(f^\ast\GG)_1\to\GG_1$ gives us a smooth functor
$f^\ast\GG\to\GG$ over $f$ which is a weak equivalence.
\end{ex}

\subsection{Path-components}
Let $\GG$ be a topological groupoid.
A path-component of $\GG$ is a full topological subgroupoid $\GG'$ of $\GG$
such that $\GG'_0$ is a minimal non-empty union of path-components of the space $\GG_0$
which is $\GG$-saturated.
Path-components of $\GG$ form a partition of $\GG$.
A topological groupoid is path-connected if it is empty or it has only one path-component.
Any path-component of a topological groupoid $\GG$ is of course a path-connected
topological groupoid and any path-connected topological
subgroupoid of $\GG$ is a subset of a path-component of $\GG$.

If $\HH'$ is a path-connected topological subgroupoid of a topological groupoid $\HH$
and $\phi:\HH\to\GG$ is a continuous functor, then the image
$\phi(\HH')$ is a path-connected topological subgroupoid of $\GG$.
An analogous result holds for Morita maps as well:

\begin{prop}\label{prop-components-to-components}
Let $\GG$ and $\HH$ be topological groupoids and let
$P$ be a topological principal $\GG$-bundle over $\HH$
with the structure maps $\pi:P\to\HH_0$ and $\epsilon:P\to\GG_0$.
If $\HH'$ is a path-connected topological subgroupoid of $\HH$,
then the set $\epsilon(\pi^{-1}(\HH'_0))$ is $\GG$-saturated and the restriction
$\GG|_{\epsilon(\pi^{-1}(\HH'_0))}$ is a path-connected full topological
subgroupoid of $\GG$.
Moreover, the space
$$P|_{\HH'}=P|_{\HH'_0}=\pi^{-1}(\HH'_0)$$
is a topological principal
$\GG$-bundle over $\HH'$, a topological principal
$\GG|_{\epsilon(\pi^{-1}(\HH'_0))}$-bundle over $\HH'$
and also a topological principal
$\GG'$-bundle over $\HH'$, where
$\GG'$ is the path-component of $\GG$ with 
$\GG|_{\epsilon(\pi^{-1}(\HH'_0))}\subset\GG'$.

In particular, if $\HH'$ is a path-component of $\HH$ and if
$P$ represents a Morita equivalence form $\HH$ to $\GG$, then the restriction
$P|_{\HH'}$ represents a Morita equivalence from $\HH'$ to $\GG'$.
\end{prop}

\begin{proof}
The subspace $\pi^{-1}(\HH'_0)$ of $P$ is $\GG$-saturated
and $\epsilon:P\to\GG_0$ is $\GG$-equivariant, so the
subspace $\epsilon(\pi^{-1}(\HH'_0))$ of $\GG_0$ is also $\GG$-saturated.
Take any two points $x,x'\in\epsilon(\pi^{-1}(\HH'_0))$.
Choose $p,p'\in \pi^{-1}(\HH'_0)$ with $\epsilon(p)=x$ and $\epsilon(p')=x'$.
The subgroupoid $\HH'$ is path-connected, so we can find continuous paths
$$ \gamma_0, \gamma_1, \ldots, \gamma_n : [0,1] \to \HH'_0 $$
such that $\gamma_0(0)=\pi(p)$, $\gamma_n(1)=\pi(p')$
and such that the points $\gamma_{k-1}(1)$ and $\gamma_k(0)$ are in the same
$\HH'$-orbit for any $k=1,2,\ldots,n$.
Since the map $\pi:P\to\HH_0$ is surjective and has local sections, we may choose
$\gamma_0, \gamma_1, \ldots, \gamma_n$ so that for any $k=0,1,\ldots,n$ the image of
$\gamma_k$ lies inside an open subset $U_k$ of $\HH_0$ on which there exists
a local section $\varsigma_k:U_k\to P$ of $\pi$.
We can choose the arrows $h_1,\ldots,h_n\in\HH'$ such that
$\src(h_k)=\gamma_{k-1}(1)$ and $\trg(h_k)=\gamma_k(0)$ for any $k=1,2,\ldots,n$.
Now
$$ \varsigma_0\com\gamma_0, \varsigma_1\com\gamma_1,\ldots, \varsigma_n\com\gamma_n $$
are continuous paths in the space $\pi^{-1}(\HH'_0)$, so
$$ \epsilon\com\varsigma_0\com\gamma_0, \epsilon\com\varsigma_1\com\gamma_1,\ldots,
   \epsilon\com\varsigma_n\com\gamma_n $$
are continuous paths in the space $\epsilon(\pi^{-1}(\HH'_0))$.

There exist a unique
$g_0\in\GG(\epsilon(p),\epsilon(\varsigma_0(\gamma_0(0))))$
such that $p=\varsigma_0(\gamma_0(0))g_0$
and a unique $g_{n+1}\in\GG(\epsilon(\varsigma_n(\gamma_n(1))),\epsilon(p'))$
such that $\varsigma_n(\gamma_n(1))=p'g_{n+1}$.
Furthermore, for any $k=1,2,\ldots,n$ there exists a unique
$g_{k}\in\GG(\epsilon(\varsigma_{k-1}(\gamma_{k-1}(1))),\epsilon(\varsigma_k(\gamma_k(0))))$
such that
$$ h_k(\varsigma_{k-1}(\gamma_{k-1}(1)))= \varsigma_k(\gamma_k(0))g_k .$$
The existence of the sequence of paths
$\epsilon\com\varsigma_0\com\gamma_0, \epsilon\com\varsigma_1\com\gamma_1,\ldots,
\epsilon\com\varsigma_n\com\gamma_n$ in $\epsilon(\pi^{-1}(\HH'_0))$
therefore implies that the points $x=\epsilon(p)$ and $x'=\epsilon(p')$
are in the same path-component of
the topological groupoid $\GG|_{\epsilon(\pi^{-1}(\HH'_0))}$.
But since this is true for any points $x$ and $x'$ in $\epsilon(\pi^{-1}(\HH'_0))$,
it follows that the  topological groupoid $\GG|_{\epsilon(\pi^{-1}(\HH'_0))}$
is path-connected. The rest of the proposition is a direct
consequence of the fact that
both $\GG|_{\epsilon(\pi^{-1}(\HH'_0))}$ and $\GG'$ are full topological subgroupoids
of $\GG$.
\end{proof}

A pointed topological groupoid is a topological groupoid $\GG$ with a chosen point
$x\in\GG_0$. The point $x$ can be equivalently given by a homomorphism of topological 
groupoids $(\ast\rightrightarrows\ast)\to\GG$, where $\ast$ denotes a fixed space
with only one point. The category of pointed topological groupoids is the
coslice category
$$ \Gpd_\ast = (\ast\rightrightarrows\ast)\downarrow \Gpd $$
For any pointed topological groupoid $(\GG;x)$ denote by
$\Pc(\GG;x)$
the path-component of $\GG$ with $x\in \Pc(\GG;x)_0$. This definition clearly
extends to a functor $\Pc:\Gpd_\ast \to \Gpd_\ast$.

Any Morita map $(\ast\rightrightarrows\ast)\to\GG$ is also determined
by a point $x\in\GG_0$, but two points give the same Morita map if, and only if, they
lie in the same $\GG$-orbit. However, any $\GG$-orbit is a subset of the space
of objects of a path-component of $\GG$. Together with Proposition
\ref{prop-components-to-components} this implies that the functor $\Pc$ is well-defined
on the coslice category of the Morita category,
$\Pc:(\ast\rightrightarrows\ast)\downarrow \GPD 
 \to (\ast\rightrightarrows\ast)\downarrow \GPD$.

Analogous results holds for smooth principal bundles;
note that a path-com\-po\-nent of a $C^\infty$-groupoid
is an open $C^\infty$-subgroupoid.

\subsection{Representations of Lie groupoids}
Let $\GG$ be a Lie groupoid. A representation
of $\GG$ (of finite rank) is a smooth complex vector bundle $\xi:\vV\to\GG_0$,
together with a homomorphism of Lie groupoids
$$ r:\GG\to\GL(\vV) $$
which is the identity on objects. Sometimes
we will denote such a representation simply by $\vV$.
Such a representation can be equivalently viewed as a smooth
left action of $\GG$ on $\vV$, $(g,v)\mapsto gv=r(g)v$,
along the map $\xi$, which is fiber-wise linear. Note that
the fiber $\vV_x$ of $\vV$
is a representation of the Lie group $\I_x(\GG)$,
for any $x\in\GG_0$.
A homomorphism between representations of $\GG$ is 
a homomorphism of complex vector bundles over $\GG_0$ which is
$\GG$-equivariant.

A unitary representation of $\GG$ is a representation $r:\GG\to\GL(\vV)$
of $\GG$ such that $\vV$ is equipped with a Hermitian metric and
$r(\GG)\subset\U(\vV)$.

A representation $r:\GG\to\GL(\vV)$ is
faithful if it is injective.
This is true if, and only if, the fiber $\vV_x$ is a faithful representation of the
Lie group $\I_x(\GG)$, for all $x\in\GG_0$.
A Lie groupoid is faithfully (unitary) representable if it admits a faithful (unitary)
representation.
A Lie groupoid is locally faithfully (unitary) representable if
there exists an open cover $(U_i)_{i\in I}$ of $\GG_0$ such that the restriction
$\GG|_{U_i}$ is faithfully (unitary) representable for all $i\in I$.
Note that any locally faithfully representable Lie groupoid is
necessarily  Hausdorff because $\GL(\vV)$ is Hausdorff.

Let $P$ be a smooth principal $\GG$-bundle over another Lie groupoid $\HH$.
For any representation $\vV$ of $\GG$ there is the induced representation
$P^\ast\vV$ of $\HH$. It is the quotient of the diagonal action of $\GG$
on the pull-back vector bundle $P\times_{\GG_0}^{\epsilon,\xi}\vV$,
with the obvious left action of $\HH$. In particular, if
$\phi:\HH\to\GG$ is a smooth functor, then $\langle \phi\rangle^\ast \vV$
is given by the usual pull-back $\phi_0^\ast\vV$.
For another Lie groupoid $\KK$ and a smooth principal $\HH$-bundle $Q$ over $\KK$
one can check that there is a natural isomorphism between the representations
$(Q\otimes_\HH P)^\ast\vV$ and $Q^\ast(P^\ast\vV)$.
Furthermore, for any $p\in P$ the representation $P^\ast(\vV)_{\pi(p)}$ of
the Lie group $\I_{\pi(p)}(\HH)$
is naturally isomorphic to the composition of the representation
$\I_{\epsilon(p)}(\GG)\to\GL(\vV_{\epsilon(p)})$ with the homomorphism
$P_p:\I_{\pi(p)}(\HH)\to\I_{\epsilon(p)}(\GG)$.

Suppose that $P$ is a Morita equivalence. This implies that 
$P_p:\I_{\pi(p)}(\HH)\to\I_{\epsilon(p)}(\GG)$ is an isomorphism of Lie groups,
for any $p\in P$. In particular, if
the representation $\vV$ is faithful,
then the representation $P^\ast\vV$ is faithful.
Hence, it follows that the faithful representability is
a Morita-invariant property of Lie groupoids:
if two Lie groupoids are Morita equivalent,
then one is faithfully representable if, and only if, the other is.
It is not difficult to show that the local faithful representability
is a Morita-invariant property as well.

\begin{ex}\rm
(1)
Let $G$ be matrix Lie group. In other words, $G$ is a Lie subgroup of the
general linear group $\GL(k,\CC)$, for a suitable natural number $k$.
Suppose that we have a smooth right action of $G$ on
a smooth manifold $M$. Then the associated action groupoid
$M\rtimes G$ over $M$ is faithfully representable. Indeed, we can
take the trivial vector bundle $\vV=M\times\CC^k$ over $M$
and define the left action of $M\rtimes G$
on $\vV$ by $(y,A)(yA,v)=(y,Av)$.
Analogously, the action groupoid associated to a smooth left action
of a matrix Lie group on a smooth manifold is faithfully representable.

Since any compact Lie group is isomorphic to a Lie subgroup of the unitary group
$\U(k)$, for some natural number $k$, it follows that
the action Lie groupoid associated to a smooth action of a compact Lie group
is faithfully unitary representable.

(2)
A Lie groupoid $\GG$ is proper if it is a smooth manifold
and the map $(\src,\trg):\GG_1\to\GG_0\times\GG_0$ is proper.
For example, the action Lie groupoid of a smooth action
of a compact Lie group on a smooth manifold is a proper
Lie groupoid. Properness is a Morita-invariant property of Lie groupoids.

Any faithfully unitary representable proper Lie groupoid $\GG$ is 
Morita equivalent to the action Lie groupoid of a smooth action
of a unitary group on a smooth manifold.
The proof of this well-known fact, essentially given by Satake \cite{Satake}
for effective orbifolds,
can be found in \cite{Kalisnik} and goes as follows:
Let $r:\GG\to\U(\vV)$ be a faithful unitary representation
and let $k$ be the rank of $\vV$.
Then $\GG$ acts freely from the left on the unitary frame bundle $\UFr(\vV)$
of $\vV$, so that the orbit space $M=\GG\setminus\UFr(\vV)$ is a smooth manifold.
The natural right action of $\U(k)$ on $\UFr(\vV)$ induces a right
action on $M$, and the Lie groupoid $\GG$
is Morita equivalent to the action Lie groupoid $M\rtimes \U(k)$.

Any proper Lie groupoid is locally linearizable and hence
locally Morita equivalent to the action Lie groupoid
of an action of a compact Lie group (see \cite{CrainicStruchiner,Weinstein,Zung}).
This implies that any proper Lie groupoid is
locally faithfully unitary representable.
\end{ex}

\section{Monodromy group}

Let $\GG$ be a topological groupoid. 
The union $\I(\GG)=\cup\{\I_x(\GG) \setsep x\in \GG_0\}$
of the isotropy groups of $\GG$
is a topological bundle of topological groups over $\GG_0$ and a normal
topological subgroupoid of $\GG$. 
The groupoid $\GG$ acts continuously on $\I(\GG)$ by conjugation:
for any $g\in\GG$ we have the associated isomorphism
of topological groups $\con_g:\I_{\src(g)}(\GG)\to\I_{\trg(g)}(\GG)$ given by
$$ \con_g(h)=ghg^{-1}. $$
We say that a function $\tau$, defined on $\I(\GG)$, is
\emph{$\GG$-conjugation-invariant}
if it is constant along the $\GG$-conjugation classes, that is, if
$\tau(ghg^{-1}) = \tau(h)$ for all $g\in\GG$ and all
$h\in\I_{\src(g)}(\GG)$.

Let $K$ be a locally compact Hausdorff topological group.
A homomorphism of topological groupoids
$\varphi:K\to\GG$ is given by a point $\omega(\varphi)$ in $\GG_0$
and a continuous homomorphism of groups
$\varphi:K\to\I_{\omega(\varphi)}(\GG)$.
The set
$$ \hom(K,\GG) $$
of all such homomorphisms is a subspace of
the space $C(K,\GG)$ of
all continuous maps from $K$ to $\GG$ with the compact-open topology.
With respect to this topology, the map
$\omega:\hom(K,\GG)\to \GG_0$ is continuous.
There is a natural continuous left $\GG$-action on $\hom(K,\GG)$
along $\omega$ defined by conjugation: 
for any $\varphi\in\hom(K,\GG)$ and any $g\in\GG(\omega(\varphi),\oo)$,
the homomorphism $g\varphi\in\hom(K,\GG)$ is given by
$$ (g \varphi)(k)=(\con_g\com \varphi)(k)=g\varphi(k)g^{-1} $$
and satisfies $\omega(g\varphi)=\trg(g)$.
We have the associated action groupoid
$$ \GG^K = \GG\ltimes\hom(K,\GG), $$
while the $\GG$-equivariant map
$\omega:\hom(K,\GG)\to \GG_0$ induces
a continuous functor of topological groupoids
$\omega=\omega_\GG:\GG^K \to \GG$.

For any $\rho\in\hom(K,\GG)$ we write
$$\hom(K,\GG;\rho)$$
for the subspace of $\hom(K,\GG)$
which is the space of objects of the path-component $\Pc(\GG^K;\rho)$
of the topological groupoid $\GG^K$
with $\rho\in\hom(K,\GG;\rho)$.
In other words, the space $\hom(K,\GG;\rho)$
is the minimal union of path-components of the space
$\hom(K,\GG)$ which includes
the point $\rho$ and
is saturated for the $\GG$-action.
The corresponding path-component of $\GG^K$ is therefore
the action groupoid
$$\Pc(\GG^K;\rho)=\GG\ltimes\hom(K,\GG;\rho).$$

If $L$ is another locally compact Hausdorff topological group,
the composition gives us a continuous map
$$ \hom(K,\GG) \times \hom(L,K) \to \hom(L,\GG). $$
For any $\alpha\in\hom(L,K)$, the continuous map
$\hom(K,\GG)\to \hom(L,\GG)$, $\varphi\mapsto \varphi\com\alpha$,
is $\GG$-equivariant, which implies
$\hom(K,\GG;\rho) \com \alpha \subset \hom(L,\GG;\rho\com\alpha)$.
In particular, if $\alpha$ is an isomorphism, then
$$ \hom(K,\GG;\rho) \com \alpha = \hom(L,\GG;\rho\com\alpha).$$
The topological group $\Aut(K)$ of automorphisms of the topological group $K$
acts continuously on $\hom(K,\GG)$ by composition from the right,
and this action commutes with the left $\GG$-action.
In particular, the group $\Aut(K)$ also acts continuously
on the action groupoid $\GG^K$
by groupoid isomorphisms.

We define the \emph{monodromy group} of the topological groupoid
$\GG$ at $\rho$ to be the subgroup
$$ \Mon(\GG;\rho) =
\{ \mu\in\Aut(K) \setsep \hom(K,\GG;\rho) \com\mu = \hom(K,\GG;\rho) \} $$
of the group $\Aut(K)$. It follows that
the monodromy group $\Mon(\GG;\rho)$ acts continuously 
on the space $\hom(K,\GG;\rho)$ and also
on the topological groupoid  $\Pc(\GG^K;\rho)$
by groupoid isomorphisms.
Let us also denote
$$ \Sigma(\GG;\rho)=\omega(\hom(K,\GG;\rho)).$$
Since the left $\GG$-action on $\hom(K,\GG;\rho)$ commutes
with the right $\Mon(\GG;\rho)$-action, it induces a left $\GG$-action
on the quotient
$\hom(K,\GG;\rho)/\Mon(\GG;\rho)$, so that the quotient projection
$\hom(K,\GG;\rho)\to\hom(K,\GG;\rho)/\Mon(\GG;\rho)$ is $\GG$-equivariant.
The map $\omega:\hom(K,\GG;\rho)\to\GG_0$ factors as a surjective map
$\hom(K,\GG;\rho)/\Mon(\GG;\rho)\to\Sigma(\GG;\rho)$
which is also $\GG$-invariant.

If $\GG$ is a Lie groupoid and $K$ a Lie group, then
$\I(\GG)$ is a topological bundle of Lie groups over $\GG_0$ and
$\GG$ acts on $\I(\GG)$ by isomorphisms of Lie groups.
In this case, all the homomorphisms in $\hom(K,\GG)$
are homomorphisms of Lie groups
and all the automorphisms in $\Aut(K)$ are automorphisms
of the Lie group $K$. Indeed, recall that
any continuous group homomorphism between Lie groups
is necessarily smooth and hence a homomorphism of Lie groups.

\begin{prop}\label{prop-monodromy-first-properties}
Let $\GG$ be a topological groupoid,
let $K$ be a locally compact Hausdorff topological group
and let $\rho\in\hom(K,\GG)$.

(i) For any $\varphi\in\hom(K,\GG;\rho)$ we have
$\Mon(\GG;\varphi)=\Mon(\GG;\rho)$.

(ii) For any $g\in\GG(\omega(\rho),\oo)$ we have
$\Mon(\GG;\con_g\com\rho)=\Mon(\GG;\rho)$.

(iii) For any continuous isomorphism $\alpha:L\to K$
between topological groups we have
$\Mon(\GG;\rho\com\alpha)=\alpha^{-1}\com\Mon(\GG;\rho)\com\alpha$.

(iv) If $\rho$ is a monomorphism, then
$\rho\com\Mon(\GG;\rho)=\Mon(\GG;\id_{\rho(K)})\com\rho$.

(v) The group $\Inn(K)$ of inner automorphisms of $K$
is a normal subgroup the monodromy group $\Mon(\GG;\rho)$.
\end{prop}

\begin{proof}
Part (i) follows from the equality
$\hom(K,\GG;\varphi)=\hom(K,\GG;\rho)$ and the definition of the monodromy group.
Part (ii) is a special case of part (i), as
the space $\hom(K,\GG;\rho)$ is saturated for the conjugation and therefore
$\con_g\com\rho\in\hom(K,\GG;\rho)$.
Observe that for any $\mu\in\Mon(\GG;\rho)$ we have
\begin{align*}
\hom(L,\GG;\rho\com\alpha) \com (\alpha^{-1}\com\mu\com\alpha)
&= \hom(K,\GG;\rho)\com (\mu\com\alpha) \\
&= \hom(K,\GG;\rho)\com\alpha
 = \hom(L,\GG;\rho\com\alpha),
\end{align*}
which yields
$\alpha^{-1}\com\Mon(\GG;\rho)\com\alpha\subset \Mon(\GG;\rho\com\alpha)$.
The equality (iii) now follows by symmetry and
directly implies (iv).
To prove part (v),
observe that
$$ \hom(K,\GG;\rho)=\hom(K,\GG;\con_{\rho(k)}\com\rho)
= \hom(K,\GG;\rho\com\con_{k})
= \hom(K,\GG;\rho)\com\con_{k}
$$
for any $k\in K$, which yields $\Inn(K)\subset\Mon(\GG;\rho)$.
The group $\Inn(K)$ is in fact normal in $\Aut(K)$, so it is normal in
$\Mon(\GG;\rho)$ as well.
\end{proof}

It is well known that for a compact Lie group $K$, the group $\Aut(K)$
is a Lie group and that the subgroup $\Inn(K)$ is open in $\Aut(K)$
(for example, see \cite{Hochschild,MontgomeryZippin}).
Proposition \ref{prop-monodromy-first-properties}(v) thus immediately implies:

\begin{prop}\label{prop-monodromy-second-properties}
Let $\GG$ be a topological groupoid,
let $K$ be a compact Lie group
and let $\rho\in\hom(K,\GG)$.
Then the monodromy group $\Mon(\GG;\rho)$ is
an open subgroup of the Lie group $\Aut(K)$.
\end{prop}

\subsection{Functoriality}
Let $K$ be a locally compact Hausdorff topological group,
let $\GG$ and $\HH$ be topological groupoids and let $P$ be
a topological principal $\GG$-bundle over $\HH$.
The space
$$ P^K = P\times_{\HH_0}^{\pi,\omega} \hom(K,\HH) $$
is a topological principal $\GG^K$-bundle over the groupoid $\HH^K$,
with the structure maps 
$P^K\to\hom(K,\HH)$, $(p,\vartheta)\mapsto \vartheta$,
and $P^K\to\hom(K,\GG)$, $(p,\vartheta)\mapsto P_p\com\vartheta$,
with the left $\HH^K$-action given by
$(h,\vartheta)(p,\vartheta)=(hp,\con_h\com\vartheta)$
and with the right $\GG^K$-action given by
$(p,\vartheta)(g,\con_{g^{-1}}\com P_{p}\com\vartheta)=(pg,\vartheta)$
(see \cite{JelencMrcun}).
This definition gives us a functor
$$ (-)^K:\GPD\to\GPD .$$
The group $\Aut(K)$ acts continuously on the space $P^K$
in the obvious way from the right such that
both the structure maps
$P^K\to\hom(K,\HH)$
and $P^K\to\hom(K,\GG)$, the left $\HH^K$-action and the right
$\GG^K$-action are $\Aut(K)$-equivariant.
There is a natural isomorphism of topological
principal bundles
$$P^K\otimes\langle\omega_\GG\rangle\cong
\langle\omega_\HH\rangle\otimes P.$$

Let $\sigma\in\hom(K,\HH)$, choose any $p\in P$ with
$\pi(p)=\omega(\sigma)$ and put $\rho=P_p\com\sigma$.
Then $P^K|_{\hom(K,\HH;\sigma)}$ is a topological
principal $\Pc(\GG^K;\rho)$-bundle
over the groupoid $\Pc(\HH^K;\sigma)$.
Of course, if we choose a different point $p'\in P$
with $\pi(p')=\omega(\sigma)$
and put $\rho'=P_{p'}\com\sigma$, 
there is a unique $g\in\GG(\oo,\epsilon(p))$
such that $pg=p'$, which implies that
$\rho'=P_{pg}\com\sigma=\con_{g^{-1}}\com P_p\com\sigma=\con_{g^{-1}}\com \rho$
and hence $\hom(K,\GG;\rho)=\hom(K,\GG;\rho')$.

The image of the set
$P^K|_{\hom(K,\HH;\sigma)}$ along the
structure map $P^K\to\hom(K,\GG)$
is a subset of $\hom(K,\GG;\rho)$.
In particular, for any automorphism
$\mu\in\Mon(\HH;\sigma)$ we have
$\sigma\com\mu\in\hom(K,\HH;\sigma)$ and hence
$(p,\sigma\com\mu)\in P^K|_{\hom(K,\HH;\sigma)}$.
As the structure map
$P^K\to\hom(K,\GG)$ maps 
$(p,\sigma\com\mu)$ into 
$$ P_p\com\sigma\com\mu = \rho\com\mu , $$
we have $\rho\com\mu\in \hom(K,\GG;\rho)$ and
therefore $\mu\in \Mon(\GG;\rho)$. We can conclude:

\begin{prop}
Let $K$ be a locally compact Hausdorff topological group,
let $\GG$ and $\HH$ be topological groupoids and let $P$ be
a topological principal $\GG$-bundle over $\HH$.
For any $\sigma\in\hom(K,\HH)$ and any $p\in P$ with
$\pi(p)=\omega(\sigma)$  we have
$$ \Mon(\HH;\sigma)\subset \Mon(\GG; P_p\com\sigma).$$
\end{prop}

If $P$ represents a Morita equivalence, then 
$P^K$ is a Morita equivalence from $\HH^K$ to $\GG^K$ and,
by Proposition \ref{prop-components-to-components},
the restriction
$P^K|_{\hom(K,\HH;\sigma)}$ is a Morita equivalence from
$\Pc(\HH^K;\sigma)$ to
$\Pc(\GG^K;\rho)$.
In this case we have
$$\Mon(\HH;\sigma)=\Mon(\GG;\rho).$$

In particular, any continuous functor $\phi:\HH\to\GG$ 
induces an $\Aut(K)$-equivariant continuous functor
$\phi^K:\HH^K\to\GG^K$ given by
$$ \phi^K(h,\vartheta)
 = (\phi(h),\hom(K,\phi)(\vartheta))
 = (\phi(h),\phi\com\vartheta)$$ 
and this gives us a functor $(-)^K:\Gpd\to\Gpd$
such that $\omega_\GG\com\phi^K = \phi\com\omega_\HH$. 
The functor $\phi^K$ restricts to a
continuous functor
$\Pc(\HH^K;\sigma) \to \Pc(\GG^K,\phi\com\sigma)$
and we have 
$$ \Mon(\HH;\sigma)\subset \Mon(\GG;\phi\com\sigma).$$
If $\phi:\HH\to\GG$ is a weak equivalence, then
$\phi^K:\HH^K\to\GG^K$ is a weak equivalence,
its restriction
$\Pc(\HH^K;\sigma) \to \Pc(\GG^K,\phi\com\sigma)$
is also a weak equivalence and
$$ \Mon(\HH;\sigma)= \Mon(\GG;\phi\com\sigma).$$

We may in fact consider the category of $K$-pointed
topological groupoids, which is the coslice category
$$ \Gpd_K=(K\rightrightarrows\ast)\downarrow \Gpd.$$
The assignment $(\oo)^K$ gives us a functor
$(\oo)^K:\Gpd_K\to\Gpd_\ast$, hence the composition
$$ \Pc\com (\oo)^K:\Gpd_K\to\Gpd_\ast $$
is also a functor.
A Morita map from $(K\rightrightarrows\ast)$ to $\GG$
corresponds to a $\GG$-orbit in $\hom(K,\GG)$, so
the assignment $(\oo)^K$ also induces a functor
$$ (\oo)^K:(K\rightrightarrows\ast)\downarrow\GPD
\to(\ast\rightrightarrows\ast)\downarrow\GPD $$
and the composition $\Pc\com (\oo)^K$ is also a functor
$(K\rightrightarrows\ast)\downarrow\GPD
\to(\ast\rightrightarrows\ast)\downarrow\GPD$.
Furthermore, the assignment of the monodromy group can be seen
as a functor
$$ \Mon : \Gpd_K \to \Grp $$
to the category $\Grp$ of groups,
which also factors as a functor
$$ \Mon : (K\rightrightarrows\ast)\downarrow\GPD \to \Grp. $$
The image of these two functors in fact lies in the lattice
of subgroups of the group $\Aut(K)$ ordered by the inclusions.

\begin{prop}\label{prop-restriction-sigma-along-weak-equivalence}
Let $K$ be a locally compact Hausdorff topological group,
let $\GG$ be a topological groupoid and let $X$ be a subspace of $\GG_0$
such that the inclusion $\GG|_X\to\GG$ is a weak equivalence.
Then for any $\rho\in\hom(K,\GG)$ with $\omega(\rho)\in X$ we have
$$ \hom(K,\GG|_X;\rho)=\hom(K,\GG;\rho) \cap \omega^{-1}(X), $$
$$ \hom(K,\GG;\rho)=\sat_\GG(\hom(K,\GG|_X;\rho)), $$
$$\Sigma(\GG|_X;\rho)=\Sigma(\GG;\rho) \cap X,$$
$$\Sigma(\GG;\rho)=\sat_\GG(\GG|_X;\rho) .$$
\end{prop}

\begin{proof}
It is straightforward that
$\hom(K,\GG|_X;\rho)\subset\hom(K,\GG;\rho) \cap \omega^{-1}(X)$.
Recall that the inclusion $\GG|_X\to\GG$ is a weak equivalence if, and only if,
the restriction $\trg|_{\src^{-1}(X)}:\src^{-1}(X)\to\GG_0$ is surjective
with local sections.
Since $\GG|_X\to\GG$ is a weak equivalence,
the induced functor
$$\Pc(\GG|_X;\rho)\to\Pc(\GG;\rho)$$
is also a weak equivalence.
Hence for any $\varphi\in\hom(K,\GG;\rho)$ there exist $\psi\in\hom(K,\GG|_X;\rho)$
and an arrow $g\in\GG(\omega(\psi),\omega(\varphi))$ such that $\varphi=\con_g\com\psi$.
If in addition we have $\omega(\varphi)\in X$, then $g\in\GG|_X$ and therefore
$\varphi\in\hom(K,\GG|_X;\rho)$. This proves the first two equations.
The last two equations are direct consequences.
\end{proof}

\section{Monodromy and representations}
We shall now study the relation between monodromy groups and
representability of Lie groupoids. In particular we will show that
a monodromy group can present an obstruction for
a Lie groupoid to be faithfully representable.

\begin{prop}\label{prop-constant-character}
Let $\GG$ be a Lie groupoid, let $K$ be a compact Lie group
and let $\rho\in\hom(K,\GG)$. 
If $\tau:\I(\GG)\to\CC$ is a continuous $\GG$-conjugation-invariant
function such that the restriction
$\tau|_{\I_x(\GG)}$ is a character of the Lie group $\I_x(\GG)$,
for all $x\in  \GG_0$, then
$$ \tau(\varphi(k))=\tau(\rho(k)) $$
for any $\varphi\in\hom(K,\GG;\rho)$ and any $k\in K$.
\end{prop}

\begin{proof}
For any $\varphi\in\hom(K,\GG)$ write
$\theta(\varphi)= \tau\com\varphi$.
This gives us a continuous map $\theta$ from $\hom(K,\GG)$ to the space
$C(K,\CC)$ of continuous complex functions on $K$
with compact-open topology.
The image $\theta(\varphi)$
is a character of the Lie group $K$,
because it is the composition of the Lie group homomorphism
$\varphi:K\to\I_{\omega(\varphi)}(\GG)$ with the character
$\tau|_{\I_{\omega(\varphi)}(\GG)}$ of the Lie group
$\I_{\omega(\varphi)}(\GG)$.

The orthogonality relations imply that
the characters of the compact Lie group $K$
form a discrete subset of the space $C(K,\CC)$
with respect to the $L^2$-norm given by the Haar measure.
The compact-open
topology on $C(K,\CC)$ coincides with the topology of
uniform convergence, which is induced by the sup-norm.
Since the $L^2$-norm is dominated by the sup-norm
on $C(K,\CC)$, the characters of $K$
form a discrete subset of $C(K,\CC)$ with respect to
the compact-open topology as well.

We may conclude that the continuous function
$\theta:\hom(K,\GG)\to C(K,\CC)$ has a discrete image
and is therefore constant on each path-component
of $\hom(K,\GG)$.
On the other hand, the map $\theta$ is also constant along the
$\GG$-orbits. Indeed, for any $\varphi\in\hom(K,\GG)$,
any $g\in\GG(\omega(\varphi),\oo)$ and any $k\in K$ we have
$$ \theta(g\varphi)(k) = \tau((g\varphi)(k))
= \tau(g\varphi(k)g^{-1})
= \tau(\varphi(k))
= \theta(\varphi)(k).
$$
It follows that $\theta$ is constant on
$\hom(K,\GG;\rho)$.
\end{proof}

\begin{prop}\label{prop-monomorphism-on-component}
Let $\GG$ be a locally faithfully representable
Lie groupoid and let $K$ be a compact Lie group.
If $\rho\in\hom(K,\GG)$ is a monomorphism, then
any $\varphi\in\hom(K,\GG;\rho)$ is
a monomorphism.
\end{prop}

\begin{proof}
First, assume that $\GG$ is faithfully representable
and choose a faithful representation $r:\GG\to\GL(\vV)$
on a complex vector bundle $\vV$ over $\GG_0$.
Let
$\tau:\I(\GG)\to\CC$ be given by $\tau(h)=\tr(r(h))$.
By Proposition \ref{prop-constant-character}
we have $\tau\com\varphi=\tau\com\rho$
for any $\varphi\in\hom(K,\GG;\rho)$.
Thus, the representations $r\com\rho:K\to\GL(\vV_{\omega(\rho})$
and $r\com\varphi:K\to\GL(\vV_{\omega(\varphi})$
have the same character, hence they are equivalent.
Since $\rho$ is a monomorphism, the representation
$r\com\rho$ is faithful, so the equivalent representation
$r\com\varphi$ is faithful as well. This implies that
$\varphi$ is a monomorphism.

Now let $\GG$ be a general locally faithfully representable Lie groupoid.
We can find an open cover $(U_j)_{j\in J}$ of $\GG_0$ such that
the restrictions $\GG|_{U_j}$ are all faithfully representable.
For any $\varphi\in\hom(K,\GG;\rho)$ we can find a finite sequence
$\varphi_0,\varphi_1,\ldots,\varphi_{n-1},\varphi_n$
in $\hom(K,\GG;\rho)$ with
$\varphi_0=\rho$ and $\varphi_n=\varphi$
such that for any $k=1,2,\ldots,n$ there exists
$j_k\in J$ satisfying $\omega(\varphi_{k-1}),\omega(\varphi_k)\in U_{j_k}$
and
$$ \varphi_k\in\hom(K,\GG|_{U_{j_k}} ; \varphi_{k-1}). $$
From the first part of the proof we know that
$\varphi_k$ is a monomorphism if $\varphi_{k-1}$ is.
Since $\varphi_0$ is a monomorphism, it follows
recursively that $\varphi_n$ is also a monomorphism.
\end{proof}

\begin{cor}\label{cor-characterization-of-monodromy-cosets}
Let $\GG$ be a locally faithfully representable Lie groupoid,
let $K$ be a compact Lie group and
let $\rho\in\hom(K,\GG)$ be a monomorphism.

(i)
The action of
the monodromy group $\Mon(\GG;\rho)$ on $\hom(K,\GG;\rho)$
is free.

(ii)
For any $\varphi,\psi\in\hom(K,\GG;\rho)$ we have
$\varphi\com\Mon(\GG;\rho)=\psi\com\Mon(\GG;\rho)$ if, and only if,
$\varphi(K)=\psi(K)$.

(iii)
If $\sigma\in\hom(K,\GG)$ is a monomorphism with
$\omega(\sigma)=\omega(\rho)$ such that
the subgroup $\sigma(K)$ is conjugated to the subgroup
$\rho(K)$ in $\I_{\omega(\rho)}(\GG)$, then
$$ \Sigma(\GG;\sigma) = \Sigma(\GG;\rho) $$
and there exists an automorphism $\mu\in\Aut(K)$
so that
$$ \hom(K,\GG;\sigma) = \hom(K,\GG;\rho)\com\mu.$$
\end{cor}

\begin{proof}
Part (i)
follows directly from Proposition \ref{prop-monomorphism-on-component}.

(ii)
If $\varphi(K)=\psi(K)$, there exists $\kappa\in\Aut(K)$ such that
$\varphi\com\kappa=\psi$. Since $\varphi,\psi\in\hom(K,\GG;\rho)$,
it follows that $\hom(K,\GG;\rho)\com\kappa=\hom(K,\GG;\rho)$ and therefore
$\kappa\in\Mon(\GG;\rho)$.

(iii) Let $g$ be an element of
$\I_{\omega(\rho)}(\GG)$ such that $\sigma(K)=\con_g(\rho(K))$.
There exists a unique automorphism $\mu\in\Aut(K)$ so that
$\con_g\com\rho\com\mu=\sigma$, and this implies
\begin{equation*} \hom(K,\GG;\sigma)=\hom(K,\GG;c_g\com\rho\com\mu)
=\hom(K,\GG;\rho)\com\mu.\qedhere
\end{equation*}
\end{proof}

\begin{theo}\label{theo-representability}
Let $\GG$ be a faithfully representable Lie groupoid, let $K$ be a compact Lie group and
let $\rho\in\hom(K,\GG)$ be a monomorphism.
Then there exists a faithful character $\chi$ of the Lie group $K$
such that 
$$ \chi\com\mu = \chi $$
for any $\mu\in\Mon(\GG;\rho)$.
\end{theo}

\begin{proof}
Choose a faithful representation $r:\GG\to\GL(\vV)$
on a complex vector bundle $\vV$ over $\GG_0$.
Let
$\tau:\I(\GG)\to\CC$ be the given by $\tau(h)=\tr(r(h))$.
The composition $\chi=\tau\com\rho$ is a faithful
character of the Lie group $K$.
For any  $\mu\in\Mon(\GG;\rho)$ we have
$\rho\com\mu\in\hom(K,\GG;\rho)$, by definition of the monodromy group.
Proposition \ref{prop-constant-character} implies
$\chi\com\mu=\tau\com\rho\com\mu=\tau\com\rho=\chi$.
\end{proof}

This theorem generalizes the obstruction to faithful representability
which was given in \cite{JelencMrcun} for proper regular Lie groupoids.
As a consequence it allows us to construct new examples of Lie groupoids
which are not faithfully representable:

\begin{ex}\rm\label{ex-flat-bundle}
Let $M$ be a connected smooth manifold with a free proper
right action of a discrete group $H$, and let $\GG$ be a faithfully representable
Lie groupoid equipped with a left action of the group $H$
by smooth functors. The left $H$-action on $\GG$ is equivalently given
by a homomorphism of groups
$$ \eta:H\to\Aut_{\Gpd}(\GG).$$
The group $H$ then also acts diagonally from the right
on the product Lie groupoid $M\times\GG$ by smooth functors.
The action on the manifold $M\times\GG_0$ is free and proper, so the space
of orbits $M\times_H\GG_0=(M\times\GG_0)/H$ is a smooth manifold.
Similarly, the space of orbits $M\times_H\GG_1=(M\times\GG_1)/H$ is
a Hausdorff $C^\infty$-space (and a smooth manifold if $\GG_1$ is a smooth manifold).
We have the quotient Lie groupoid
$$ \BB = M\times_H\GG = \big( (M\times\GG_1)/H \rightrightarrows (M\times\GG_0)/H \big) $$
which is clearly locally faithfully representable.

Let $K$ be a compact Lie subgroup of $\I_y(\GG)$, for some $y\in\GG_0$, 
choose a point $x\in M$, and
let $\rho\in\hom(K,\BB)$ be given by $\rho(k)=(x,k)H$.
Consider the subgroup $H'=\{ h\in H \,|\, \eta(h)(K)=K\}$. It follows that
for any $h\in H'$ we have $\nu(h)=\eta(h)|_K\in\Aut(K)$ and that
$\nu:H'\to\Aut(K)$ is a homomorphism of groups. We claim that
$$ \nu(H')\subset\Mon(\BB;\rho). $$
Indeed, for any $h\in H'$
we can choose a continuous path $\gamma:[0,1]\to M$ from $x$ to $xh$ and lift this path
to a continuous path $\bar{\gamma}:[0,1]\to\hom(K,\BB)$ by
$$ \bar{\gamma}(t)(k)=(\gamma(t),k)H. $$
Observe that $\bar{\gamma}(0)=\rho$ and that $\bar{\gamma}(1)=\rho\com\nu(h)$,
as
$$ \bar{\gamma}(1)(k)=(xh,k)H=(x,\eta(h)(k))H
=(x,\nu(h)(k))H=\rho(\nu(h)(k)).$$

If the Lie groupoid $\BB=M\times_H\GG$ is faithfully representable,
Theorem \ref{theo-representability}
implies that there exists a faithful character $\chi$ of the Lie group $K$
such that
$$ \chi\com\nu(h)=\chi $$
for any $h\in H'$.
However, this may be impossible, in particular if $K$ is abelian and
$\nu(h)$ is of infinite multiplicative order. In this case, the Lie groupoid
$M\times_H\GG$ is an example of a locally faithfully representable
Lie groupoid which is not (globally) faithfully representable.

As a concrete example, let
$\GG$ be the product $\HH\times T^n=\HH\times (T^n\rightrightarrows\ast)$
of an arbitrary non-empty proper Lie groupoid $\HH$
with the $n$-torus $T^n$, for some $n\geq 2$,
and we choose $M$, $H$ and $\eta$ so that for some $h\in H$ and some $y\in\HH_0$,
the automorphism $\eta(h)$ of $\HH\times T^n$
restricts to an automorphism of the subgroup
$K=\{1_y\}\times T^n\subset \I_y(\HH)\times T^n=\I_{(y,\ast)}(\HH\times T^n)$
of infinite multiplicative order.
The associated groupoid $M \times_H (\HH\times T^n)$
is in this case a proper Lie groupoid which is not faithfully representable.

Note that we can choose $\HH$ so that $\BB$ in not regular.
This example generalizes the one given in \cite{JelencMrcun}, where
$\GG=T^n$ and the proper Lie groupoid $\BB$ is regular, in fact a bundle
of compact Lie groups. Before that, only a special case with
$n=2$ and $M=\RR$ was known
(see \cite{LuckOliver,Trentinaglia}).
\end{ex}

There is a simple, but important special case of Example \ref{ex-flat-bundle}
in which the constructed Lie groupoid
is necessarily faithfully representable:

\begin{prop}
Let $M$ be a smooth connected
manifold with a free proper right action of a discrete group $H$,
let $G$ be a Lie group and let $q:H\to G$ be a homomorphism of groups.
The composition of the homomorphism $q$ with
the conjugation homomorphism $\con:G\to\Aut(G)$, $g\mapsto\con_g$,
defines an action of $H$ on $G$.

(i)
If the Lie group $G$ is faithfully representable,
then the associated Lie groupoid $M\times_H G$ is faithfully representable.

(ii)
If the Lie group $G$ is faithfully unitary representable,
then the Lie groupoid $M\times_H G$ is faithfully unitary representable.
\end{prop}

\begin{proof}\label{prop-simple-example-bundle}
(i)
Since the group $G$ is faithfully representable,
we may assume, without loss of generality, that $G$ is a Lie subgroup
of the Lie group  $\GL(m,\CC)$. Write $\iota:G\to\GL(m,\CC)$ for the inclusion.
The image of the homomorphism $\con\com q: H\to\Aut(G)$
lies inside the subgroup $\Inn(G)\subset\Aut(G)$ of inner automorphisms
of the group $G$.
$$
\xymatrix{
  & G \ar[d]^{\con} \ar[r]^-{\iota} & \GL(m,\CC) \\
H \ar[ru]^{q} \ar[r]_-{\con\com q} & \Inn(G) &
}
$$
We have the left action of $H$ on $\CC^m$ given
by the composition $\iota\com q:H\to\GL(m,\CC)$, and with
this action we construct
the vector bundle
$$ M\times_H \CC^m = (M\times\CC^m)/H $$
over the smooth manifold $M/H$.
The inclusion $\iota$ induces a faithful representation
of the Lie groupoid $M\times_H G$ on the vector bundle $ M\times_H \CC^m$.

(ii)
If $G$ is faithfully unitary representable,
we can replace the group $\GL(m,\CC)$ in the proof of (i)
with the unitary group $\U(m)$.
The left $H$-action on $\CC^m$ is now unitary
and there is a natural Hermitian metric on the vector bundle 
$M\times_H \CC^m$.
\end{proof}

In fact, in a special case we can show that the property
given in Theorem \ref{theo-representability} is not only necessary,
but also sufficient for the given
Lie groupoid to be faithfully representable:

\begin{theo}\label{theo-bundle-representability}
Let $M$ be a smooth connected
manifold with a proper free right action of a free discrete group $H$,
let $K$ be a compact Lie group
and let $\eta:H\to \Aut(K)$ be a homomorphism of groups.
Then the proper Lie groupoid $M\times_H K$
is faithfully unitary representable if, and only if, there exists a faithful character
$\chi$ of $K$ such that $\chi\com\eta(h)=\chi$ for all $h\in H$.
\end{theo}

\begin{proof}
One implication is a direct consequence of Theorem \ref{theo-representability}.
To prove the other implication, assume that $\chi$
is a faithful character of $K$
such that $\chi\com\eta(h)=\chi$ for all $h\in H$, and choose
a faithful unitary
representation $r:K\to\U(m)$ with character $\chi$.
As a consequence, the unitary representations
$r\com\eta(h)$ and $r$ are unitary equivalent, which means that there exists
an element of the group $\U(m)$ which intertwines these two representations.
In particular,
this element lies of the normalizer $G=\N_{\U(m)}(r(K))$ 
of the subgroup $r(K)$ in the group $\U(m)$.
We have the inclusion $\iota:G\to\U(m)$ and the homomorphism of groups
$\nu:\Inn(G)\to\Aut(K)$ given by the composition of the restriction to $r(K)$
with the isomorphism $\Aut(K)\cong\Aut(r(K))$ induced by $r:K\to r(K)$.

Since the group $H$ is free, we can choose
a homomorphism of groups $q:H\to G$ such that
the diagram
$$
\xymatrix{
 & G \ar[d]^{\con} \ar[r]^{\iota} & \U(m) \\
H \ar[ru]^{q} \ar[r]^-{\con\com q} \ar[rd]_-{\eta} & \Inn(G) \ar[d]^{\nu} &  \\
 & \Aut(K) &
}
$$
commutes, or explicitly, such that for any $h\in H$ and any $k\in K$ we have
$$ r(\eta(h)(k)) = \con_{q(h)}(r(k)) = q(h) r(k) q(h)^{-1}. $$
By Proposition \ref{prop-simple-example-bundle},
we obtain the associated faithful unitary representation
$$r': M\times_H G \to \U(M\times_H \CC^m)$$
of the Lie groupoid $M\times_H G$.
The monomorphism $r:K\to G$ induces a monomorphism
of Lie groupoids $M\times_H K\to M\times_H G$, and the composition of this
monomorphism with the representation $r'$ is a faithful unitary representation
of the Lie groupoid $M\times_H K$.
\end{proof}

\section{Monodromy and proper Lie groupoids}

In this section we shall study the monodromy groups of proper Lie groupoids.
Recall that any proper Lie groupoid $\GG$
is locally faithfully representable,
so by Proposition \ref{prop-monomorphism-on-component}
all homomorphisms in $\hom(K,\GG;\rho)$ are monomorphisms,
for any compact Lie group $K$
and any monomorphism $\rho\in\hom(K,\GG)$. In particular,
the monodromy group $\Mon(\GG;\rho)$ acts
freely on $\hom(K,\GG;\rho)$.

\subsection{Compact Lie groups}\label{subsec-compact-Lie-groups}\label{subsec-Lie-groups}
Let $K$ and $G$ be compact Lie groups and let
$\sigma:K\to G$ be a homomorphism of Lie groups. Recall that the space
$\hom(K,G)$ 
is Hausdorff and that the $G$-conjugation orbits
are open in $\hom(K,G)$ \cite{ConnerFloyd,MontgomeryZippin}.
This implies that the space $\hom(K,G;\sigma)$ is precisely
the $G$-conjugation orbit through $\sigma$,
$$ \hom(K,G;\sigma)=G\sigma=\Inn(G)\com\sigma, $$
which is homeomorphic to the manifold
$G/\C_{G}(\sigma(K))$ of left cosets of the centralizer $\C_{G}(\sigma(K))$
of the subgroup $\sigma(K)$ in $G$. In particular, the space
$\hom(K,G;\sigma)$ has a structure of a compact smooth manifold.

Assume now that $\sigma$ is a monomorphism.
For an automorphism $\mu\in\Aut(K)$ we have $G\sigma\com\mu=G\sigma$ if,
and only if, there exists $g\in G$ such that
$\sigma\com\mu=\con_g\com\sigma$. If this is the case, then $g$ lies in
the normalizer $\N_{G}(\sigma(K))$ of the subgroup
$\sigma(K)$ in $G$. 
It follows that we have a natural epimorphism
$$ \N_{G}(\sigma(K)) \to \Mon(G;\sigma) $$
which maps an element $g\in\N_{G}(\sigma(K))$
to the automorphism $\sigma^{-1}\com \con_g|_{\sigma(K)}\com \sigma$
of the Lie group $K$.
The kernel of this epimorphism is exactly the centralizer
$\C_{G}(\sigma(K))$ of the subgroup $\sigma(K)$ in $G$.
Thus, we have constructed a natural isomorphism of groups
$$ \Mon(G;\sigma) \cong \N_{G}(\sigma(K))/\C_{G}(\sigma(K)).$$

If we choose $K=G$ and $\rho=\id$, we have
$$ \hom(G,G;\id)=G\id=\Inn(G) \cong G/\C_G(G)$$
and also
$$ \Mon(G;\id)=\Inn(G) \cong G/\C_G(G),$$
where $\C_{G}(G)$ is the center of the Lie group $G$.

\subsection{Actions of compact Lie groups}
Let $G$ be a compact Lie group and let
$M$ be a smooth manifold with a smooth left $G$-action.
We write
$G_x$ for the isotropy group of the $G$-action at a point $x\in M$.
For a subgroup $H$ of $G$, we denote by
$$ \fix_H(M) = M^{H}=\{ y\in M \,|\, H\subset G_y \} \subset M $$
the fixed point space of the subgroup $H\subset G$ in $M$, and by 
$$ \fix_H(M;x) $$
the path-component of the space $\fix_H(M)$ which includes
a point $x\in\fix_H(M)$. The fixed point space $\fix_H(M)$
is a closed subspace of $M$. If the subgroup $H$ is compact,
the Bochner linearization theorem \cite{DuistermaatKolk} implies that
the fixed point space $\fix_H(M)$ is locally a manifold,
so $\fix_H(M)$ is locally path-connected and
each path-component of $\fix_H(M)$ is a closed embedded submanifold
of $M$. However, different path-components of $\fix_H(M)$
may have different dimensions. 

We consider the action groupoid $G\ltimes M$, which is a proper
Lie groupoid over the manifold $M$.
Let $K$ be a compact Lie group. Note that the conjugation
action of the Lie groupoid $G\ltimes M$ on the space
$\hom(K,G\ltimes M)$ along $\omega:\hom(K,G\ltimes M)\to M$
is in this case simply given by the conjugation action
of the Lie group $G$ on the space $\hom(K,G\ltimes M)$, and that
the map $\omega$ is $G$-equivariant.
Any homomorphism $\varphi\in\hom(K,G\ltimes M)$
is of the form $\varphi(k)=(\vartheta(k),x)$, for a point $x\in M$ and
a homomorphism of Lie groups $\vartheta:K\to G_x$.
Conversely,
for any homomorphism $\sigma:K\to G$ of Lie groups and
any point $y\in\fix_{\sigma(K)}(M)$ there is a homomorphism
$$ \bar{\sigma}(y)\in\hom(K,G\ltimes M) $$
defined by $\bar{\sigma}(y)(k)=(\sigma(k),y)$, and this
gives us the partially defined continuous section
$$ \bar{\sigma}: \fix_{\sigma(K)}(M)\to\hom(K,G\ltimes M) $$
of the map $\omega:\hom(K,G\ltimes M)\to M$.

For any $x\in M$
we have the sequence of homomorphisms of Lie groupoids
$$ G_x \to  G\ltimes M \to G, $$
where the first homomorphism $\iota_x:G_x\to G\ltimes M$
is given by inclusion of the
isotropy group $G_x$ at the point $x$, $\iota_x(g)=(g,x)$,
while the second homomorphism is
the projection $\pr_G:G\ltimes M\to G$. These homomorphisms induce the sequence
of inclusions
$$ \Mon(G_x;\sigma)
   \subset
   \Mon(G\ltimes M;\bar{\sigma}(x))
   \subset 
   \Mon(G;\sigma) ,$$
for any homomorphism of Lie groups $\sigma:K\to G_x$.

\begin{ex}\rm\label{ex-action-with-fixed-point}
Let $K$ and $G$ be compact Lie groups, let $M$ be a smooth manifold with
a smooth left $G$-action, and let 
$\sigma:K\to G$ be a homomorphism of Lie groups.
It the $G$-action on $M$ has a fixed point $x\in M$, then
$\Mon(G\ltimes M;\bar{\sigma}(x)) = \Mon(G;\sigma).$
\end{ex}

Choose a homomorphism of Lie groups
$\sigma:K\to G$. 
Since the projection $\pr_G:G\ltimes M\to G$ is a homomorphism
of Lie groupoids, the induced continuous map
$$ \hom(K,\pr_G):\hom(K,G\ltimes M)\to\hom(K,G),\;\;\;\;\;\;\;\;
   \varphi\mapsto\pr_G\com\varphi, $$
is $G$-equivariant.
The orbit $G\sigma$ is $G$-invariant,
open and closed
in $\hom(K,G)$, thus the inverse image
$\hom(K,\pr_G)^{-1}(G\sigma)$ is $G$-saturated,
open and closed in the space $\hom(K,G\ltimes M)$.
Observe that we have
$$ \hom(K,\pr_G)^{-1}(\{\sigma\}) = \bar{\sigma}(\fix_{\sigma(K)}(M))
   \subset \hom(K,\pr_G)^{-1}(G\sigma).$$
We define a map
$$ \hat{\sigma}: G\times \fix_{\sigma(K)}(M) \to \hom(K,\pr_G)^{-1}(G\sigma) $$
by
$$ \hat{\sigma}(g,y)(k)=(\con_g(\sigma(k)),gy).$$
It is clear that this map is continuous, surjective and $G$-equivariant
with respect to the left $G$-action on $G\times \fix_{\sigma(K)}(M)$
given by $g'(g,y)=(g'g,y)$.
The centralizer
$$ H=\C_{G}(\sigma(K)) $$
is a compact subgroup of the normalizer
$\N_{G}(\sigma(K))$ and acts by multiplication from the right
on $G$, but also from the left on $\fix_{\sigma(K)}(M)$ because
$\N_{G}(\sigma(K)) \fix_{\sigma(K)}(M) = \fix_{\sigma(K)}(M)$.
We combine these two actions to get a smooth
right diagonal $H$-action on the space $G\times \fix_{\sigma(K)}(M)$.
One may check that the map $\hat{\sigma}$ factors through the quotient projection
$G\times \fix_{\sigma(K)}(M) \to G\times_H \fix_{\sigma(K)}(M)
=\big(G\times \fix_{\sigma(K)}(M)\big)/H$ as 
a continuous surjective $G$-equivariant map
$$ \tilde{\sigma}:G\times_H\fix_{\sigma(K)}(M) \to \hom(K,\pr_G)^{-1}(G\sigma).$$
It is straightforward to check that $\tilde{\sigma}$ is a bijection.
We claim:

\begin{lem}\label{lem-homeomorphism-hom}
For any compact Lie group $G$, any
smooth manifold $M$ with a smooth left $G$-action and
any homomorphism $\sigma:K\to G$ of Lie groups,
the map $\tilde{\sigma}$
is a $G$-equivariant homeomorphism.
\end{lem}

\begin{proof}
We have already seen that the map $\tilde{\sigma}$
is a $G$-equivariant continuous bijection. We will show 
that the map $\hat{\sigma}$ is closed, which implies that $\tilde{\sigma}$ is closed
and hence a homeomorphism.
To this end, observe first that the map
$f:G\to\hom(K,G)$, $g\mapsto \con_g\com\sigma$, is a continuous map from a compact space
to a Hausdorff space, which yields that
the map $f\times\id_M:G\times M\to \hom(K,G)\times M$ is closed.
Since the map $G\times M\to G\times M$, $(g,y)\mapsto (g,gy)$, is a diffeomorphism,
the composition of $f\times\id_M$ with this diffeomorphism is also closed.
In particular, the restriction of this composition to the closed subspace
$G\times \fix_{\sigma(K)}(M)$ of $G\times M$ is a closed map
$\beta:G\times \fix_{\sigma(K)}(M)\to \hom(K,G)\times M$, given by
$\beta(g,y)=(\con_g\com\sigma,gy)$.
Now observe that we have a commutative diagram of continuous maps
$$
\xymatrix{
G\times \fix_{\sigma(K)}(M) \ar[r]^\beta \ar[dr]_{\hat{\sigma}} & \hom(K,G)\times M \\
& \hom(K,\pr_G)^{-1}(G\sigma) \ar[u]_{(\hom(K,\pr_G),\omega)}
}
$$
in which $\beta$ is closed and $(\hom(K,\pr_G),\omega)$ is injective.
It follows that $\hat{\sigma}$ is closed.
\end{proof}

\begin{prop}\label{prop-holonomy-of-action}
Let $K$ and $G$ be compact Lie groups, let
$M$ be a smooth manifold with a smooth left $G$-action,
let $x$ be a point in $M$ and let $\sigma:K\to G_x$ be a homomorphism of Lie groups.
Denote $H=\C_{G}(\sigma(K))$.

(i)
The space $\hom(K,G\ltimes M;\bar{\sigma}(x))$ is an open,
closed and $G$-saturated
subset of $\hom(K,G\ltimes M)$, and we have a $G$-equivariant homeomorphism
$$ G \times_{H} \big( \sat_H (\fix_{\sigma(K)}(M;x)) \big)
   \to \hom(K,G\ltimes M;\bar{\sigma}(x)) $$
given by $(g,y)H\mapsto (k\mapsto (\con_g(\sigma(k)),gy))$.
In particular, we have a natural structure of a smooth manifold
on the space $\hom(K,G\ltimes M;\bar{\sigma}(x))$.

(ii)
We have
$$
\hom(K,G\ltimes M;\bar{\sigma}(x))
= \sat_G \big( \bar{\sigma} ( \fix_{\sigma(K)}(M;x) ) \big) 
$$
and
$$ \Sigma(G\ltimes M;\bar{\sigma}(x)) = \sat_G \big( \fix_{\sigma(K)}(M;x) \big).$$
\end{prop}

\begin{proof}
(i)
The smooth manifold $\fix_{\sigma(K)}(M;x)$ is a path-component of the space
$\fix_{\sigma(K)}(M)$ which is locally a manifold, so
the saturation $\sat_H (\fix_{\sigma(K)}(M;x))$ is also a smooth
manifold and a union of path-components of the space $\fix_{\sigma(K)}(M)$.
It follows that the product $G\times \sat_H (\fix_{\sigma(K)}(M;x))$
is a smooth manifold which is $H$-saturated,
and since $H$ is a compact Lie group acting smoothly and freely on
$G\times \sat_H (\fix_{\sigma(K)}(M;x))$, the space of orbits
$G\times_H \sat_H (\fix_{\sigma(K)}(M;x))$ is a smooth manifold as well.
Furthermore, the subset $G\times_H \sat_H (\fix_{\sigma(K)}(M;x))$
is clearly the minimal union of path-components of the space
$G\times_H \fix_{\sigma(K)}(M)$ which is $G$-saturated
and includes the point $(1,x)H$.
The $G$-equivariant map $\tilde{\sigma}$,
which is a homeomorphism by Lemma \ref{lem-homeomorphism-hom},
therefore maps $G \times_{H} \sat_H (\fix_{\sigma(K)}(M;x))$ onto
the space $\hom(K,G\ltimes M;\bar{\sigma}(x))$.

(ii)
First equality is a direct consequence of part (i),
while the second follows from the first because the map
$\omega:\hom(K,G\ltimes M)\to M$ is $G$-equivariant.
\end{proof}

\begin{ex}\rm\label{ex-linear-action}
Let $G$ be a compact Lie group with an orthogonal representation
$G\to \Ort(n)$. This means that we have an orthogonal left action
of $G$ on the vector space $\RR^n$.
Let $M$ be an open ball in $\RR^n$ centered at the origin $0\in \RR^n$.
Since $M$ is $G$-saturated, the action of $G$ on $\RR^n$ restricts to
an action on $M$.

Let $K$ be a compact Lie group and let $\sigma:K\to G$ be a monomorphism.
The origin $0$ is a fixed point of the $G$-action, so we have
$$ \Mon(G\ltimes M;\bar{\sigma}(0)) = \Mon(G;\sigma)
   \cong \N_{G}(\sigma(K))/\C_{G}(\sigma(K)),$$
by Example \ref{ex-action-with-fixed-point} and Subsection \ref{subsec-Lie-groups}.
The fixed point set
$\fix_{\sigma(K)}(\RR^n)$ is a linear subspace of $\RR^n$, so the
fixed point space
$\fix_{\sigma(K)}(M)=\fix_{\sigma(K)}(\RR^n)\cap M$
is an open ball in the space
$\fix_{\sigma(K)}(\RR^n)$. In particular we have
$$ \fix_{\sigma(K)}(M;0)=\fix_{\sigma(K)}(M)=\fix_{\sigma(K)}(\RR^n)\cap M.$$
It follows that
$$ \hom(K,G\ltimes M;\bar{\sigma}(0))
 = \sat_G \big( \bar{\sigma} ( \fix_{\sigma(K)}(\RR^n)\cap M ) \big)
$$
and
$$ \Sigma(G\ltimes M;\bar{\sigma}(0)) = 
   \sat_G \big( \fix_{\sigma(K)}(\RR^n)\cap M \big) $$
by Proposition \ref{prop-holonomy-of-action}.
\end{ex}

\subsection{Proper groupoids}
Let $\GG$ be a proper Lie groupoid. 
Recall that such a groupoid is locally Morita equivalent to the action Lie groupoid
of an action of a compact Lie group. This is a consequence of
the Weinstein-Zung linearization theorem for proper Lie groupoids
\cite{CrainicStruchiner,Weinstein,Zung}.

In fact, for any point $x\in\GG_0$ we can choose a slice $S$ through $x$,
which is an embedded submanifold of $\GG_0$ of dimension
$m=\dim(\GG_0) -\dim(\GG x)$ such that
$S\cap \GG x =\{x\}$ and so that for any $y\in S$ the orbit
$\GG y$ is transversal to $S$ at $y$.
This implies that the restriction $\GG|_S$ is a Lie groupoid, that
$U=\sat_\GG(S)$ is an open subset of $\GG_0$
and that the embedding $\GG|_S \to \GG|_U$ is a weak equivalence. 
Furthermore, by linearization theorem
we can choose $S$ so that
there exist 
an orthogonal action of the isotropy group $G=\I_x(\GG)$ on $\RR^m$,
and open ball $M\subset\RR^m$ centered at the origin
and an isomorphism of Lie groupoids
$G\ltimes M \to \GG|_S$
which maps the origin $0\in M$ to the point $x$ and which is the projection
when restricted to the isotropy group
$\I_x(G\ltimes M)=G\times\{0\}$.
If we compose this isomorphism with the inclusion $\GG|_S \to \GG|_U$, we obtain
a weak equivalence
$$ \lambda:G\ltimes M \to \GG|_U .$$

Let $K$ be a compact Lie group, let $\sigma:K\to G$
be a monomorphism of Lie groups and let
$\rho=\bar{\sigma}(0)\in\hom(K,G\ltimes M)$. 
Note that the composition
$\lambda\com\rho\in\hom(K,\GG|_U)$  is also a monomorphism.
By Example \ref{ex-linear-action} and
Proposition \ref{prop-restriction-sigma-along-weak-equivalence}
we therefore conclude that
$$ \hom(K,\GG|_U;\lambda\com\rho)
= \sat_\GG \big( \lambda \com 
  \big( \bar{\sigma} ( \fix_{\sigma(K)}(\RR^n)\cap M ) \big) \big) $$
and
$$ \Sigma(\GG|_U;\lambda\com\rho)
= \sat_\GG \big( \lambda \big( \fix_{\sigma(K)}(\RR^n)\cap M \big) \big), $$
while
$$ \Mon(\GG|_U;\lambda\com\rho)
= \Mon(G;\sigma) \cong \N_{G}(\sigma(K)) / \C_{G}(\sigma(K)). $$

\subsection{Proper regular groupoids}
Recall that a Lie groupoid is regular if all its isotropy groups
have the same dimension.
For more details and a discussion on classification of such Lie groupoids,
see \cite{Moerdijk2003}.

Let $\GG$ be a path connected regular
proper Lie groupoid.
The orbits of the left $\GG$-action on $\GG_0$
are leafs of a regular foliation of the manifold
$\GG_0$ and we have the induced representation of $\GG$
on the normal bundle $\N\GG$ of this foliation. The kernel of this
representation is a locally trivial bundle $\K(\GG)$ of compact Lie groups over $\GG_0$,
a subbundle of the isotropy bundle $\I(\GG)$.
The fibers of the bundle $\K(\GG)$ are called the ineffective isotropy groups of $\GG$.

Choose a point $x\in \GG_0$, put $K=\K(\GG)_x$ and let
$\rho:K=\K(\GG)_x\to\GG$ be the inclusion.
By \cite[Proposition 11]{JelencMrcun} it follows that
$$ \hom(K,\GG;\rho) = \hom(K,\K(\GG);\rho) $$
and any element of $\hom(K,\GG;\rho)$
is an isomorphism onto an ineffective isotropy group.
In particular, the
map $\omega:\hom(K,\GG;\rho)\to\GG_0$ is
a locally trivial bundle and
$$ \Sigma(\GG;\rho)=\GG_0.$$
The right $\Mon(\GG;\rho)$-action on $\hom(K,\GG;\rho)$ is free,
but also transitive along the fibers of $\omega:\hom(K,\GG;\rho)\to\GG_0$,
by Corollary \ref{cor-characterization-of-monodromy-cosets},
and the map
$$ \omega:\hom(K,\GG;\rho)\to\GG_0 $$
is a $\GG$-equivariant principal $\Mon(\GG;\rho)$-bundle over $\GG_0$.
The associated functor
$$ \Pc(\GG^K;\rho)=\GG\ltimes \hom(K,\GG;\rho) \to \GG $$
is a Serre fibration of Lie groupoids \cite{JelencMrcun}
and hence there is a natural long exact sequence \cite[Theorem 12]{JelencMrcun}
of homotopy groups:
\begin{equation*}
\begin{split}
\ldots\to\pi_n(\Mon(\GG;\rho))\to\pi_n(\Pc(\GG^K;\rho))
\stackrel{\pi_{n}(\omega)}{\to}&\pi_n(\GG)
\stackrel{\partial}{\to}\pi_{n-1}(\Mon(\GG;\rho))\to\ldots \\
\ldots\to\pi_1(\Mon(\GG;\rho))\to\pi_1(\Pc(\GG^K;\rho))
\stackrel{\pi_{1}(\omega)}{\to}&\pi_1(\GG)
\stackrel{\partial}{\to}\pi_{0}(\Mon(\GG;\rho))
\end{split}
\end{equation*}
Note that for any $n\geq 1$ we have
$$ \pi_n(\Pc(\GG^K;\rho))=\pi_n(\GG\ltimes \hom(K,\GG;\rho))=\pi_n(\GG^K;\rho) $$
and
$$ \pi_n(\Mon(\GG;\rho))=\pi_n(\Aut(K))=\pi_n(\Inn(K)) $$
by Proposition \ref{prop-monodromy-second-properties}.
Furthermore, the homomorphism of groups
$$ \partial:
\pi_1(\GG) \to \pi_{0}(\Mon(\GG;\rho)), $$
which we call the monodromy map, is surjective.

\begin{ex}\rm
Let $M$ be a connected smooth manifold with a free proper
right action of a discrete group $H$, let $G$ be a compact Lie group
and let $\eta:H\to\Aut(G)$ be a homomorphism of groups.
As in Example \ref{ex-flat-bundle} we have the associated
quotient Lie groupoid
$$ \BB = M\times_H G $$
which is in this case a regular proper Lie groupoid, in fact a locally trivial
bundle of Lie groups. Furthermore, we have
$\K(\BB)=\I(\BB)=\BB$. 
We have a natural isomorphism
$$ \hom(G,\BB)\cong M\times_H \hom(G,G) $$
of Lie groupoids over $M/H$,
where the left action of $H$ on $\hom(G,G)$ is given by $\eta$ and the composition,
that is, for any $h\in H$ and any $\varphi\in\hom(G,G)$ we have
$h\varphi=\eta(h)\com\varphi$.
Let $x$ be a point of $M$ and let $\rho\in\hom(G,\BB)$ be the monomorphism
given by $\rho(g)=(x,g)H$. Under the isomorphism
$\hom(G,\BB)\cong M\times_H \hom(G,G)$, the monomorphism $\rho$
corresponds to the $H$-orbit of the pair $(x,\id_G)$.

For any $\varphi\in\hom(G,G)$,
the set $G\varphi=\Inn(G)\com\varphi$ is open, closed and $G$-saturated in $\hom(G,G)$.
Note that $\eta(H)\com\Inn(G)=\Inn(G)\com\eta(H)$ because
the subgroup $\Inn(G)$ is normal in $\Aut(G)$. In particular,
the product of the subgroups $\eta(H)$ and $\Inn(G)$ in $\Aut(G)$
is a subgroup
$$ L=\eta(H)\com\Inn(G) $$
of the group $\Aut(G)$. It follows that
the set $L\com\varphi$ is open and closed $\hom(G,G)$, and is also
the minimal subset of $\hom(G,G)$ which
includes $\varphi$ and
is both $G$-saturated and $H$-saturated.

This implies that the subset
$$ M\times_H L \subset  M\times_H \hom(G,G) $$
is open, closed and $\BB$-saturated, and that
the natural isomorphism $\hom(G,\BB)\cong M\times_H \hom(G,G)$
restricts to the isomorphism
$$ \hom(G,\BB;\rho)) \cong M\times_H L. $$
In particular, this implies 
that the monodromy group $\Mon(\BB;\rho)$ is the product
of the subgroups $\eta(H)$ and $\Inn(G)$ in $\Aut(G)$,
$$ \Mon(\BB;\rho) = \eta(H)\com \Inn(G).$$
\end{ex}


\begin{thebibliography}{99}

\bibitem{ConnerFloyd}
P.~E.~Conner and E.~E.~Floyd.
\newblock {\em Differentiable periodic maps.}
\newblock Springer-Verlag, Berlin, 1964.

\bibitem{CrainicStruchiner}
Marius~Crainic and Ivan~Struchiner.
\newblock On the linearization theorem for proper Lie groupoids.
\newblock {\em Ann. Sci. Ec. Norm. Super. (4)}  46(5):723–-746, 2013.

\bibitem{DuistermaatKolk}
J.~J.~Duistermaat and J.~A.~C.~Kolk.
\newblock {\em Lie groups.}
\newblock Springer-Verlag, Berlin, 2000.

\bibitem{HenriquesMetzler}
Andre~Henriques and David~S.~Metzler.
\newblock Presentations of noneffective orbifolds.
\newblock {\em Trans. Amer. Math. Soc.} 356(6):2481--2499, 2004.

\bibitem{Hochschild}
G.~Hochschild.
\newblock The automorphism group of a Lie group.
\newblock {\em Trans. Amer. Math. Soc.} 72(2):209--216, 1952.

\bibitem{Kalisnik}
J.~Kali\v{s}nik.
\newblock Representations of orbifold groupoids.
\newblock {\em Topology Appl.} 155(11):1175--1188, 2008.

\bibitem{LuckOliver}
Wolfgang~L{\"u}ck and Bob~Oliver.
\newblock The completion theorem in {$K$}-theory for proper actions of a
          discrete group.
\newblock {\em Topology} 40(3):585--616, 2001.

\bibitem{Mackenzie}
Kirill~C.~H.~Mackenzie.
\newblock {\em General theory of Lie groupoids and Lie algebroids.}
\newblock Cambridge Univ. Press, Cambridge, 2005.

\bibitem{Moerdijk2002}
I.~Moerdijk.
\newblock Orbifolds as groupoids: an introduction.
\newblock Orbifolds in mathematics and physics (Madison 2001).
\newblock {\em Contemp. Math.}  310:205--222, 2002.

\bibitem{Moerdijk2003}
I.~Moerdijk.
\newblock Lie groupoids, gerbes, and non-abelian cohomology.
\newblock {\em $K$-Theory} 28(3):207--258, 2003.

\bibitem{MoerdijkMrcun2003}
I.~Moerdijk and J.~Mr\v{c}un.
\newblock {\em Introduction to foliations and Lie groupoids.}
\newblock Cambridge Univ. Press, Cambridge, 2003.

\bibitem{MoerdijkMrcun2005}
I.~Moerdijk and J.~Mr{\v{c}}un.
\newblock Lie groupoids, sheaves and cohomology.
\newblock Poisson geometry, deformation quantisation and group representations.
\newblock {\em London Math. Soc. Lecture Note Ser.} 323:145--272, 2005.

\bibitem{MontgomeryZippin}
Deane~Montgomery and Leo~Zippin.
\newblock A theorem on Lie groups.
\newblock {\em Bull. Amer. Math. Soc.} 48:448--452, 1942. 

\bibitem{Mrcun1996}
Janez~Mr\v{c}un.
\newblock {\em Stability and invariants of Hilsum-Skandalis maps.}
\newblock PhD thesis, Utrecht University, 1996.
\newblock {\em arXiv:} math/ 0506484v1.

\bibitem{JelencMrcun}
B.~Jelenc and J.~Mr\v{c}un.
\newblock  Homotopy sequence of a topological groupoid with a basegroup
and an obstruction to presentability of proper regular Lie groupoids.
\newblock {\em J. Homotopy Relat. Struct.} 10(3):519–-536, 2015. 

\bibitem{Satake}
I.~Satake.
\newblock On a generalization of the notion of manifold.
\newblock {\em Proc. Nat. Acad. Sci. U.S.A.} 42:359--363, 1956.

\bibitem{Trentinaglia} 
Giorgio~Trentinaglia.
\newblock On the role of effective representations of Lie groupoids.
\newblock {\em Adv. Math.} 225(2):826--858, 2010.

\bibitem{Weinstein}
Alan~Weinstein.
\newblock Linearization of regular proper groupoids.
\newblock {\em J. Inst. Math. Jussieu} 1(3):493--511, 2002.

\bibitem{Zung}
Nguyen~Tien~Zung.
\newblock Proper groupoids and momentum maps: linearization, affinity, and convexity.
\newblock Ann. Sci. \'Ecole Norm. Sup. 39(5):841–-869, 2006.


\end{thebibliography}
\end{document}